\documentclass[a4paper,twoside,10pt]{article}

\usepackage[utf8]{inputenc}  
\usepackage[T1]{fontenc}   
\usepackage{times}
\usepackage{amsmath,amsfonts,amssymb,amsthm,euscript,pifont}
\usepackage{mathrsfs}
\usepackage{color}
\usepackage{soul}
\usepackage[colorlinks=true]{hyperref}
\usepackage{fullpage}
\usepackage{xspace}
\usepackage{bbm}
\usepackage{authblk}

\usepackage{dsfont}

\newcommand\indi[1]{{\mathds{1}}_{#1}}

\newtheorem{theorem}{Theorem}[section]

\newtheorem{proposition}[theorem]{Proposition}

\newtheorem{remark}[theorem]{Remark}
\newtheorem{hypothesis}[theorem]{Hypothesis}

\newcommand{\er}{\mathbb{R}}

\newcommand{\EE}{\mathbb{E}}
\newcommand{\Ww}{\mathbb{W}}

\newcommand{\bx}{\boldsymbol{q} }
\newcommand{\tbx}{\widetilde{\boldsymbol{q}}}
\newcommand{\tR}{\widetilde{R}}
\newcommand{\ty}{\widetilde{y}}
\newcommand{\tV}{\widetilde{V}}
\newcommand{\dyb}{\overline{\delta y}}

\newcommand{\bW}{ \boldsymbol{W} }

\newcommand{\PP}{\mathcal{P}}

\usepackage{upgreek}
\newcommand{\mes}{{(\pi)}}

\numberwithin{equation}{section}

\title{Clarification and complement to ``Mean-field description and
propagation of chaos in networks of Hodgkin-Huxley and FitzHugh-Nagumo neurons''}

\author[1]{Mireille Bossy}
\author[2]{Olivier Faugeras}
\author[1]{Denis Talay}
\affil[1]{TOSCA Laboratory, INRIA Sophia Antipolis -- M\'editerran\'ee, France}
\affil[2]{NeuroMathComp Laboratory, INRIA Sophia Antipolis -- M\'editerran\'ee, France}

\date{\today}

\begin{document}
\maketitle
\begin{abstract}
In this note, we clarify the well-posedness of the limit equations to the mean-field $N$-neuron  models proposed in \cite{baladron-al} and we prove the associated propagation of chaos property. We also complete the modeling issue in \cite{baladron-al} by discussing the well-posedness of the stochastic differential equations which govern the behavior of the ion channels and the amount of available neurotransmitters.
\end{abstract} 



\section{Introduction}\label{sec:Intro}
 
The paper Baladron et al.~\cite{baladron-al} studies quite general networks of neurons and aims to prove that  these networks propagate chaos in the sense originally developed by Sznitman~\cite{sznitman} after the seminal work of Kac on mean field limits and McKean's works~\cite{McKean} on diffusion processes propagating chaos. As observed by the authors, the membrane potentials of the neurons in the networks they consider are described by interacting stochastic particle dynamics. 
The coefficients of these McKean--Vlasov systems  
are not globally Lipschitz. Therefore the
classical results of the propagation of chaos theory do not directly apply and a specific analysis needs to be performed. The main theorems (existence and uniqueness of the limit system when the 
number of particles tends to infinity, propagation of chaos property) 
are stated under a fairly general hypothesis on the coefficients.
Unfortunately the proof in \cite{baladron-al} p. 24-25 involves an erroneous management of hitting times in combination with a truncation technique, and the limit system may not be well-defined under the too general hypothesis used by the authors. Indeed, the following 
equation, where $\phi$ is a bounded  and locally Lipschitz function and $Z_0$ is a random variable, satisfies the hypothesis
made in \cite{baladron-al} p.15-16: 
$$Z_t = Z_0 + \int_0^t \EE \phi(Z_s) ~ds. $$ 
However Scheutzow exhibited examples of a function $\phi$ and initial condition $Z_0$  for which many solutions do exist: see Counterexample 2 in ~\cite{scheutzow} and the remark 
which follows it \footnote{\normalsize Similarly, Counterexample 1 in \cite{scheutzow} contradicts  the results on neuronal models claimed at the end of Section 1 and in Appendix B of the paper~\cite{touboul:14}.}.

This note restricts the neuron model to the much used variants of the FitzHugh-Nagumo and Hodgkin-Huxley models. Our objective is two-fold:
first, we discuss a modeling issue on the diffusion coefficients
of the equations describing the proportions of open and closed
channels that guarantees that these variables do not
escape from the interval $[0,1]$. This was not completely achieved
in~\cite{baladron-al} and can be seen as a complement to this paper.

Second, we give a rigorous proof of the propagation of chaos property.

\section{The models}\label{sec:the-models}
In this section we present and discuss the stochastic models considered in  Baladron et al.~\cite{baladron-al} for the electrical
activity of $\bar{p}$ populations of neurons.  Each population has a label $\alpha$ and $N_\alpha$ elements.
	We denote by $\PP $ the set of the $\bar{p}$ population labels and by $N:=\sum_{\alpha\in\PP } N_\alpha$ the total number of neurons.

Given the neuron $i$ in a population $\alpha$, 
the stochastic time evolution of the membrane potential is denoted by $(V^{i}_t)$. In the case of the Hodgkin-Huxley model, the Sodium and Potassium activation variables, which represent proportions of open gates along  the neuron~$i$
are respectively denoted by $(n^{i}_t)$, $(m^{i}_t)$. 
The Sodium inactivation variable, which is also a proportion of open gates, is denoted by
$(h^{i}_t)$. 
In the case of the Fitzhugh-Nagumo model, the recovery variable is denoted by $(w^i_t)$. Both models feature synaptic variables $(y^{i}_t)$ which represent the proportion of available neurotransmitters at the synapses of neuron $i$.

The synaptic connections between neurons are assumed to be chemical in \cite{baladron-al}. We make the same assumption here.
This implies that the synaptic current $I_{ij}^{\text{syn}}$ from the presynaptic neuron $j$  in population $\gamma$ to the
postsynaptic neuron $i$ in population $\alpha$ writes
\[
I_{ij}^{\text{syn}}(t)=-g_{ij}(t)(V^{i}_t-\overline{V}^{\alpha\gamma}),
\]
where $\overline{V}^{\alpha\gamma}$ is the synaptic reversal potential of the $j \to i$ synapse, assumed to be approximately constant across populations, and $g_{ij}(t)$ the electric conductance of that synapse.  Hence the postsynaptic neuron $i$ belongs to population $\alpha$ and the presynapic neuron $j$ to population $\gamma$. This conductance is the product of the maximal conductance, noted $J^{ij}_t$,  of the synapse by the proportion $y^{j}_t$ of neurotransmitters available at neuron $j$. Conductances are positive quantities.

The processes $(V^i_t,n^i_t,m^i_t,h^i_t,w^i_t,y^i_t)$ are defined by means of the stochastic differential
systems~\eqref{eq:N-neuron-simple-conductance}
or~\eqref{eq:N-neuron-sign-preserving} in the $N$-neuron model section below. The mean-field limit processes are defined in \eqref{eq:N-neuron-sign-preserving-limit} and \eqref{eq:N-neuron-simple-conductance-limit}. 
Well-posedness of those systems is postponed to Section~\ref{sec:systemes-bien-poses}.

\subsection{The $N$-neuron model}

The variants of the FitzHugh-Nagumo and Hodgin-Huxley dynamics
proposed in Baladron et al.~\cite{baladron-al} to model neuron networks are all of the two types below; the only differences concern
the algebraic expressions of the function $F_\alpha$ and the fact that
the Fitzhugh-Nagumo model does not depend on the variables
$(n_t^{i},m_t^{i},h_t^{i})$ but on the
recovery variable $(w^{i}_t)$ only. Conversely, the Hodgkin-Huxley model does not depend on $w^i_t$. In what follows we denote by $\bx$ the
vector $(n,m,h)$ of $\er^3$ in the case of the Hodgkin-Huxley model, 
and the real $(w)$   in the case of the Fitzhugh-Nagumo model.
We also note 
$(\bW^{i}_t) = (W^{i,V}_t,W^{i,y}_t,W^{i,n}_t,W^{i,m}_t,W^{i,h}_t)$,
independent Brownian motions ($1\leq i\leq N$).

Given a neuron $i$, the number $p(i)=\alpha$ denotes the type of the 
population it belongs to.

The equations~(\ref{eq:N-neuron-simple-conductance}) and
(\ref{eq:N-neuron-sign-preserving}) below are those studied in
\cite{baladron-al}. They correspond to two different models for the
maximum conductances. The first one does not respect the positivity
constraint while the second one guarantees that these quantities stay
positive.
In these equations, all quantities which
are not time indexed are constant parameters.

\paragraph{Simple maximum conductance variation.} ~

For $i$ and $j$ such that $p(i) = \alpha$ and $p(j)=\gamma$, 
the model assumes that $J^{ij}_t$ fluctuates around a mean $\overline{J}^{\alpha\gamma}$ according to a white noise with standard deviation $\sigma^J_{\alpha \gamma}$:
\[
J^{ij}_t=\overline{J}^{\alpha\gamma}+\sigma^J_{\alpha \gamma}\dfrac{dB^{i \gamma}_t}{dt}.
\]
For $(B^{i\gamma}_t)$ a family of independent Brownian motions, independent of the Brownian family $(\bW^i_t)$, the equations describing the dynamics of the state vector of neuron $i$ in population $\alpha$ write 
\begin{equation}\label{eq:N-neuron-simple-conductance}
\left\{
\begin{aligned}
&\mbox{for $i$ such that } p(i) = \alpha,  \mbox{ for }\bx = (w) \mbox{ or } (n,m,h), \\
dV^{i}_t &= F_\alpha(t, V^{i}_t , 
\bx^{i}_t) dt   
- \sum_{ \gamma \in \PP }  (V^{i}_t - \overline{V}^{\alpha \gamma})
\dfrac{\bar{J}^{\alpha\gamma}}{N_\gamma} 
\left(\sum_{j=1}^N 
\indi{\{p(j)=\gamma\}}y^{j}_t\right) dt \\
&\quad\quad 
- \sum_{ \gamma\in\PP }  (V^{i}_t - \overline{V}^{\alpha \gamma}) 
\dfrac{\sigma^J_{\alpha\gamma}}{N_\gamma} \left( 
\sum_{j=1}^N \indi{\{p(j) = \gamma\}}y^{j}_t\right) 
d B^{\gamma,i}_t + \sigma_\alpha dW^{i,V}_t, \\
dy^{i}_t &=  \left(a^\alpha_rS_\alpha(V^{i}_t) 
(1-y^{i}_t)  - a^\alpha_dy^{i}_t \right) dt 
+ \sqrt{ \left| a^\alpha_r S_\alpha(V^{i}_t) (1-y^{i}_t)  
+ a^\alpha_dy^{i}_t \right| } \; \chi(y^{i}_t) 
dW^{i,y}_t, 
\end{aligned}
\right.
\end{equation}
coupled with
\begin{equation} \label{eq:N-neuron-w}
dw^{i}_t = c_\alpha(V^{i}_t + a_\alpha - b_\alpha 
w^{i}_t) dt
\end{equation}
or
\begin{align} \label{eq:N-neuron-x}
\begin{aligned}
dx^{i}_t  &= \left( \rho_{x}(V^{i}_t)
(1-x^{i}_t) 
- \zeta_{x}(V^{i}_t)x^{i}_t\right) dt + \sqrt{| \rho_{x}(V^{i}_t)(1-x^{i}_t) 
+ \zeta_{x}(V^{i}_t )
x^{i}_t|} \;\chi(x^{i}_t) dW^{i,x}_t \\
&\qquad\mbox{ for } x=n,m,h.
\end{aligned}
\end{align}
The reader may wonder about the reason for the square root term and the function $\chi$ in the diffusion coefficient of the SDE for the processes $x^i$ and $y^i$. 
The square root arises from the fact that this SDE is a Langevin approximation to a stochastic hybrid, or piecewise deterministic, model of the ion channels. There is a finite (albeit large) number of such ion channels and each of them can be modeled as jump Markov processes coupled to the membrane potential. Altogether ion channels and membrane potentials are described by a piecewise deterministic Markov process which, as shown for example in \cite{wainrib:06}, can be approximated by the solution to the SDE shown in \eqref{eq:N-neuron-simple-conductance} and \eqref{eq:N-neuron-sign-preserving}. Hypothesis \ref{hypo:coeff}-(i) below on the function~$\chi$ implies
	that the processes $x^i$ and $y^i$ are valued in $[0,1]$:
	See Section~\ref{sec:systemes-bien-poses}.

\paragraph{Sign-preserving maximum conductance variation.} 

\begin{equation}\label{eq:N-neuron-sign-preserving}
\left\{
\begin{aligned}
& \mbox{For $i$ such that $p(i)=\alpha$,}\mbox{ for }\bx = (w) \mbox{ or } (n,m,h), \\
dV^{i}_t &= F_\alpha(t, V^{i}_t , \bx^{i}_t) dt   
- \sum_{ \gamma\in\PP }  (V^{i}_t - \overline{V}^{\alpha \gamma})
\dfrac{J_t^{i\gamma}}{N_\gamma} \left( \sum_{j=1}^N 
\indi{\{p(j) = \gamma\}} y^{j}_t\right) dt  + \sigma_\alpha dW^{i,V}_t, \\
d J^{i \gamma}_t &= \theta_{\alpha\gamma}
\left( \bar{J}^{\alpha\gamma} - J^{i \gamma}_t \right) dt 
+ \sigma^{J}_{\alpha\gamma}\sqrt{J^{i \gamma}_t} 
dB^{i\gamma}_t \ \mbox{for all }\gamma\in\PP , \\
dy^{i}_t &=  \left(a^\alpha_rS_\alpha(V^{i}_t) 
(1-y^{i}_t)  - a^\alpha_dy^{i}_t \right) dt 
+ \sqrt{ \left|a^\alpha_rS_\alpha(V^{i}_t) (1-y^{i}_t) 
+ a^\alpha_dy^{i}_t \right| } \; \chi(y^{i}_t) 
dW^{i,y}_t, 
\end{aligned}
\right.
\end{equation}
coupled with \eqref{eq:N-neuron-w} or \eqref{eq:N-neuron-x}

\subsection{The mean-field limit models} 

When making the $N_\alpha$s tend to infinity,  the linear structure of the above $N$-neuron models w.r.t. the $(y^{i}_t)$, the linear structure of
the dynamics of the $(y^{i}_t)$, and the mutual independence of the Brownian motions $(B^{i,\gamma}_t)$, $(W^{j,y}_t)$,  lead one to mean-field dynamics. 
The limit processes $(V^\alpha_t, y^\alpha_t, n^\alpha_t, m^\alpha_t, h^\alpha_t, w^\alpha_t, \alpha\in\PP )$ are solutions to the SDEs \eqref{eq:N-neuron-simple-conductance-limit} and \eqref{eq:N-neuron-sign-preserving-limit}
where $(B^{\alpha\gamma}_t,W^{\alpha,V}_t,
W^{\alpha,n}_t,W^{\alpha,m}_t,W^{\alpha,h}_t, 
\alpha\in\PP )$ denote independent Brownian motions. 

\paragraph{Simple maximum conductance variation.}  

For all $\alpha$ in $\PP $,

\begin{equation}\label{eq:N-neuron-simple-conductance-limit}
\left\{
\begin{aligned}
dV^\alpha_t &= F_\alpha(t, V^\alpha_t , \bx^\alpha_t) dt   
- \sum_{ \gamma\in\PP }  (V^\alpha_t - \overline{V}^{\alpha \gamma})
\bar{J}^{\alpha\gamma}  \EE[y^\gamma_t]  dt  
- \sum_{ \gamma\in\PP } (V^\alpha_t - \overline{V}^{\alpha \gamma}) 
\sigma^J_{\alpha\gamma} \EE[y^\gamma_t]
d B^{\alpha\gamma}_t + \sigma_\alpha dW^{\alpha,V}_t, \\
dy^{\alpha}_t &=  \left(a^\alpha_rS_\alpha(V^{\alpha}_t) 
(1-y^{\alpha}_t)  - a^\alpha_dy^{\alpha}_t \right) dt 
+ \sqrt{ \left| a^\alpha_r S_\alpha(V^{\alpha}_t) (1-y^{\alpha}_t)  
	+ a^\alpha_dy^{\alpha}_t \right| } \; \chi(y^{\alpha}_t) 
dW^{\alpha,y}_t, 
\end{aligned}
\right.
\end{equation}
coupled with
\begin{equation}\label{eq:N-neuron-limit-w}
dw^\alpha_t = c_\alpha(V^\alpha_t + a_\alpha 
- b_\alpha w^\alpha_t) dt
\end{equation}
or
\begin{align}\label{eq:N-neuron-limit-x}
\begin{aligned}
dx^\alpha_t  
&= \left(\rho_{x}(V^\alpha_t)(1-x^\alpha_t) 
- \zeta_{x}(V^\alpha_t )x^\alpha_t\right) dt 
+ \sqrt{\displaystyle | \rho_{x}(V^\alpha_t)(1-x^\alpha_t) 
+ \zeta_{x}(V^\alpha_t )x^\alpha_t|} \;\chi(x^\alpha_t) 
dW^{\alpha,x}_t,\\
& \qquad\mbox{ for } x=n,m,h, 
\end{aligned}
\end{align}
where again $\bx^\alpha_t$ stands for $(w^{\alpha}_t)$ in the Fitzhugh-Nagumo model and  for $(n^{\alpha}_t, m^{\alpha}_t,h^\alpha_t)$ in the Hodgin-Huxley model. 

Notice that the diffusion coefficients of the
$(y^{\alpha}_t)$ play no role in the above mean--field 
dynamics since one readily sees that
\begin{align} \label{expression-overline-y-gamma}
\begin{aligned}
 \EE[y^\alpha_t] &= \EE [y_0^{\alpha}] \exp\left\{ - a^\alpha_d t -  a_r^\alpha\int_0^t  
\EE[S_\alpha(V^\alpha_\theta)] d\theta \right\}\\
& \qquad + \int_0^t a_r^\alpha \EE[S_\alpha(V^\alpha_s)]\exp\left\{ - a^\alpha_d (t-s)  -  a_r^\alpha\int_s^t  
\EE[S_\alpha(V^\alpha_\theta)] d\theta \right\} ds,  \ \mbox{for }\alpha\in\PP .
\end{aligned}
\end{align}
\paragraph{Sign-preserving maximum conductance variation.} 
With the same notation, for  all $\alpha$ in $\PP $, 
\begin{equation} \label{eq:N-neuron-sign-preserving-limit}
\left\{
\begin{aligned}
dV^\alpha_t &= F_\alpha(t, V^\alpha_t , \bx^\alpha_t) dt   
- \sum_{ \gamma\in\PP } (V^\alpha_t - \overline{V}^{\alpha \gamma}) 
 J_t^{\alpha\gamma} \EE[y_t^{\gamma}]  dt  
+ \sigma_\alpha dW^{\alpha,V}_t, \\
d J^{\alpha\gamma}_t &= \theta_{\alpha\gamma}
\left( \bar{J}^{\alpha\gamma} 
- J^{\alpha\gamma}_t \right) dt + \sigma^{J}_{\alpha\gamma}  
\sqrt{J^{\alpha\gamma}_t} dB^{\alpha\gamma}_t \ \mbox{for all }\gamma\in\PP ,  \\
dy^{\alpha}_t &=  \left(a^\alpha_rS_\alpha(V^{\alpha}_t) 
(1-y^{\alpha}_t)  - a^\alpha_dy^{\alpha}_t \right) dt 
+ \sqrt{ \left| a^\alpha_r S_\alpha(V^{\alpha}_t) (1-y^{\alpha}_t)  
	+ a^\alpha_dy^{\alpha}_t \right| } \; \chi(y^{\alpha}_t) 
dW^{\alpha,y}_t, 
\end{aligned}
\right.
\end{equation}
coupled with \eqref{eq:N-neuron-limit-w} or \eqref{eq:N-neuron-limit-x}

As in the simple maximum conductance variation case,
the diffusion coefficients of the
$(y^{\alpha}_t)$ play no role in the above mean--field 
dynamics and $\EE[y^\alpha_t]$ is given by~
\eqref{expression-overline-y-gamma}.

\subsection{Hypotheses}

Our hypotheses on the coefficients of the neuron models are the following.

\begin{hypothesis} \label{hypo:coeff}

\item{\bf \textit{(i)} On the ion channel models.}
The function $\chi$ is  bounded Lipschitz continuous with compact support included in  the interval $(0,1)$. 

The functions $\rho_{x}$, $\zeta_{x}$ are strictly positive, Lipschitz, and bounded. 

\item{\bf \textit{(ii)} On the chemical synapse model. }
The functions $S_\alpha$  are of the sigmoid type, that is, 
$S_\alpha(v) = C / (1 + \exp( -\lambda (v - \delta)))$ 
with  suitable positive parameters $C$, $\lambda$, $\delta$ depending on $\alpha$. The parameters $a^\alpha_r$, $a^\alpha_d$ are also positive.

\item{\bf \textit{(iii)} On the membrane model. }
The drift terms $F_\alpha$ are continuous, one-sided
Lipschitz  w.r.t. $v$ and Lipschitz w.r.t. $\bx$. More precisely, there 
exist a positive real number $L$ and a positive map $M(v,v')$
such that for all $t\in [0,T]$, for all $\bx$, $\bx'$  in $\er^3$ or
$\er$, $v$ , $v'$ in $\er$, 
\begin{equation} \label{cond:one-sided-lipschitz}
\begin{split}
(F_\alpha(t,v,\bx) - F_\alpha(t,v',\bx)) (v-v') 
&\leq L(v - v')^2 - M(v,v') (v - v')^2, \\
|F_\alpha(t,v,\bx) - F_\alpha(t,v,\bx')| &\leq L\|\bx-\bx'\|.
\end{split}
\end{equation}

\item{\bf \textit{(iv)} The initial conditions} 
$V^{i}_0$, $J^{i\gamma}_0$, $y^{i}_0$, $w^{i}_0$,
$n^{i}_0$, $m^{i}_0$, $h^{i}_0$
are i.i.d. random variables 
with the same law as  $V^{\alpha}_0$, $J^{\alpha\gamma}_0$, $y^{\alpha}_0$, $w^{\alpha}_0$,
$n^{\alpha}_0$, $m^{\alpha}_0$, $h^{\alpha}_0$, when $p(i)=\alpha$. 
We also assume that $V^{\alpha}_0$ and   $J^{\alpha\gamma}_0$ admit moments of any order. 

Moreover, the support of the law of $y^{\alpha}_0$ belongs to $[0,1]$, as well as the support of the laws of $n^{\alpha}_0$, $m^{\alpha}_0$, $h^{\alpha}_0$, for all $\alpha$ in $\PP $. 
The support of the law of $J^{\alpha\gamma}_0$ belongs to $(0,+\infty)$. 
\end{hypothesis}

\begin{remark}
For each neuron population $\alpha$, the function $S_\alpha$ represents the
concentrations of the transmitter
released into the synaptic cleft by a presynaptic spike.

Our hypothesis on the support of the function $\chi$ is
essential to force the proportion processes  
$(y^i_t)$, $(n^i_t)$, $(m^i_t)$, $(h^i_t)$ to live in the interval $(0,1)$.

In the case of the FitzHugh-Nagumo model,  for all $\alpha$ we have
$F_\alpha(v,\bx)  = - \tfrac{1}{3} v^3 + v -w$, so that we may choose  $L=1$ and  
$M(v,v') = \tfrac{1}{3} (|v| - |v'|)^2$.

Finally, the i.i.d hypothesis in part~(iv) is only used in 
Section~\ref{sec:systemes-bien-poses} where it allows simplifications,
but it can be relaxed to initial chaos by classical arguments in the propagation of chaos
literature.
\end{remark}

\begin{remark}
We notice that a one-sided Lipschitz condition also
appears in the work by Lu\c con and 
Stannat~\cite{L-S} for stochastic particle systems in random
environments in which they model one-population networks of Fitzhugh-Nagumo neurons. However their model does not include synaptic interactions as ours does. This has led us in particular to consider square root
diffusion coefficients in the dynamics of the synaptic variables
which, as shown below, requires specific arguments to prove that
the particle systems are well posed and propagate chaos.
\end{remark}

\begin{remark}
The boundedness of the functions $\rho_x$ and $\zeta_x$ is a
technical hypothesis which simplifies the analysis but can be relaxed,
provided that the above models have solutions which do not explode in
finite time. However this comfortable hypothesis is quite reasonable 
for neuron models since the membrane potentials take values 
between -100 mV and 100 mV. It is therefore implicitly understood 
that Lipschitz functions which reasonably fit data within 
the interval $[-100,100]$ are extended to bounded Lipschitz functions 
on the entire real line.
\end{remark}

\section{SDEs in rectangular cylinders}
\label{sec:systemes-contraints}

In the above $N$-neuron and limit models one requires that, for all
$i$, $\alpha$ and $x=n,m,h$,
the concentration processes $(x^{i}_t)$, $(x^\alpha_t)$,  
$(y^{i}_t)$ and $(y^\alpha_t)$ are well-defined
and take values in the interval $[0,1]$. Each one of these processes  is 
one-dimensional but not
Markov since the coefficients of their dynamics depend on $(V_t)$
and thus on all the components of the 
systems~(\ref{eq:N-neuron-sign-preserving})
and~(\ref{eq:N-neuron-simple-conductance}).
In this context, classical arguments for
one-dimensional Markov diffusion processes
such as Feller's tests or comparison theorems cannot be used to show 
that the processes do not exit from $[0,1]$. We thus need to
develop an ad hoc technique. Instead of focusing on the above neuron models,  we rather introduce a more general setting.

\medskip

Consider the stochastic differential equation
\begin{equation} \label{EDS-Xk-U}
\begin{cases}
\begin{split}
dX^{(1)}_t &= 
b_1(X_t,U_t)dt+A_1(X_t,U_t)dW^{(1)}_t, \\
\ldots &  \\
dX^{(k)}_t &= 
b_k(X_t,U_t)dt+A_k(X_t,U_t)dW^{(k)}_t, \\
dU_t &= \beta(X_t,U_t)dt+\Gamma(X_t,U_t)dW^\ast_t, 
\end{split}
\end{cases}
\end{equation}
where $(X_t) := (X^{(1)}_t,\ldots,X^{(k)}_t))$ is $\er^k$-valued, 
$(U_t)$ is $\er^d$-valued,  
$(W_t)= (W^{(1)}_t,\ldots,W^{(k)}_t)$ 
is a $\er^k$-valued $(\mathcal{F}_t)$-Brownian motion, and 
$(W^\ast_t)$  is a $\er^r$-valued Brownian motion.

We aim to exhibit  conditions on the coefficients 
of~(\ref{EDS-Xk-U})  which imply that the process $(X_t,U_t)$ takes 
values in the infinite rectangular cylinder $[0,1]^k\times\er^d$.

\begin{remark}
Many stochastic models of the type~(\ref{EDS-Xk-U}) which arise in
Physics need to satisfy the constraint that the process
$(X_t)$ is valued in the hypercube, say, $[0,1]^k$.
The algebraic expressions of the coefficients derived from physical
laws may be `naturally' defined  only for $x$ in $[0,1]^k$.  
However, one typically can construct continuous extensions of these
coefficients on the whole Euclidean space. These extensions may be 
arbitrarily chosen, provided that they satisfy the 
hypothesis~\ref{hypo:H} and that 
Equation~\eqref{EDS-Xk-U} has a weak solution which does not
explode in finite time.  In Section~\ref{sec:systemes-bien-poses} we 
develop this argument to show that all systems in 
Section~\ref{sec:the-models} are well-posed and that
the concentration processes are all valued in~$[0,1]$.
\end{remark}

Assume:
\begin{hypothesis}\label{hypo:H}
The locally Lipschitz continuous functions $b_i$, $A_i$, $\beta$ 
and $\Gamma$ are such that, on some filtered probability space  
$(\Omega,\mathcal{F},\mathbb{P},(\mathcal{F}_t,t\geq 0))$,
there exists a weak solution to \eqref{EDS-Xk-U} which does not explode in finite time.
In addition,
\item{(i)} The functions $A_i$, $i=1,\ldots,k$,  satisfy, for all $u$ in $\er^d$, 
$$\forall 
x_i\in\er\setminus(0,1),~~A_i(x_1,\ldots,x_{i-1},x_i,x_{i+1},\ldots,x_k,u)=0. $$

\item{(ii)} The functions $b_i$, $i=1,\ldots,k$,  satisfy, for all $u$  in $\er^d$ and  $x_1,\ldots,x_{i-1},x_{i+1},\ldots,x_k$
in $\er^{k-1}$, 
\begin{equation*}
\begin{cases}
b_i(x ,u)\geq0 \mbox{ on }\{x_i\leq0\}, \\ 
b_i(x,u)\leq 0 \mbox{ on }\{x_i \geq 1\}.
\end{cases}
\end{equation*}
\end{hypothesis}

The following argument implies that we may limit ourselves to the
case $k=1$. Set $U^\sharp:=(X^{(2)},\ldots,X^{(k)},U)$.
Then, for obviously defined new coefficients $\beta_1^\sharp$,
$A^\sharp$, etc., and a $\er^{r+k-1}$-valued Brownian motion
$W^{\ast\sharp}$ one has
\begin{equation*} 
\begin{cases}
\begin{split}
dX^{(1)}_t &= b^\sharp_1(X^{(1)}_t,U^\sharp_t)dt
+A^\sharp_1(X^{(1)}_t,U^\sharp_t)dW^{(1)}_t, \\
dU^\sharp_t &= \beta^\sharp_1(X^{(1)}_t,U^\sharp_t)dt
+\Gamma^\sharp_1(X^{(1)}_t,U^\sharp_t)dW^{\ast\sharp}_t.
\end{split}
\end{cases}
\end{equation*}
If we can prove that $X^{(1)}$ takes values in $[0,1]$, then
the same arguments would show that all the other components enjoy
the same property. We therefore consider the 
system~\eqref{EDS-Xk-U} with $k=1$.

\begin{proposition} \label{prop:main}
Suppose that $0\leq X^{(1)}_0 \leq 1$~a.s.
Under Hypothesis~\ref{hypo:H} it holds 
\[\mathbb{P}\left( \forall t \geq 0,~  0 \leq X_t^{(1)} \leq 1 
\right) = 1.\]
\end{proposition}
\begin{proof}
We limit ourselves to prove that $0\leq X_t^{(1)}$ for all $t\geq 0$ a.s. We can use very similar arguments to get the other inequality.

Let $\Psi_\epsilon$ be a positive
decreasing function of class $\mathcal{C}^2(\er)$
with $\Psi_\epsilon(x)=1$ on $(-\infty,-\epsilon]$ and $\Psi_\epsilon(x)=0$ on $[0,+\infty)$.  Let $\tau_n$ be the first time the process $(X^{(1)}_t)$ exits from the interval $(-n,n)$. As
$\Psi_\epsilon(x)=0$ on $\er_+$, $\Psi_\epsilon(X^{(1)}_0)=0$ and  It\^o's formula leads to
\begin{align*}
\Psi_\epsilon(X^{(1)}_{t\wedge\tau_n}) 
&=  \int_0^{t\wedge\tau_n} \indi{\{X^{(1)}_s\leq 0\}}
~b^\sharp_1(X^{(1)}_s,U^\sharp_s)
~\Psi_\epsilon'(X^{(1)}_s) ds \\
& \qquad +\dfrac{1}{2} 
\int_0^{t\wedge\tau_n} \indi{\{X^{(1)}_s \leq 0\}} 
(A^\sharp_1(X^{(1)}_s,U^\sharp_t))^2~\Psi_\epsilon''(X^{(1)}_s) ds \\
& \qquad + \int_0^{t\wedge\tau_n} \indi{\{X^{(1)}_s\leq 0\}}
~A^\sharp_1(X^{(1)}_s,U^\sharp_s)
~\Psi_\epsilon'(X^{(1)}_s) dW^{\ast\sharp}_s \quad =:  I_1 + I_2 + I_3.
\end{align*}
As $\Psi_\epsilon$ is decreasing, Hypothesis~\ref{hypo:H}-(ii)
and the definition of $b^\sharp_1$ imply that $I_1\leq0$.  As $\indi{\{x\leq 0\}}A^\sharp_1(x,u)=0$ for all $(x,u)$, one has
$I_2=0$. Finally, $I_3$ is a martingale. Therefore
$$ \forall t>0,~\EE\Psi_\epsilon(X^{(1)}_{t\wedge\tau_n})=0. $$
Fatou's lemma implies
$$ \forall t>0,~\EE\Psi_\epsilon(X^{(1)}_t)=0. $$
Now consider a family of functions $\Psi_\epsilon$ which 
pointwise converge to $\indi{(-\infty,0)}$ and satisfy $\sup_\epsilon|\Psi_\epsilon|_\infty=1$. 
Lebesgue's Dominated Convergence Theorem implies
$$ \forall t>0,~\EE[\indi{(-\infty,0)}(X^{(1)}_t)]=0.$$
In other words, the process $((X_t^{(1)})^{-},t\geq 0)$ is a modification of the null process. As they both are continuous processes, they are
indistinguishable (see, e.g., Karatzas and Shreve~\cite[Chap.1,Pb.1.5]{K-S}).
\end{proof}

\section{The models are well-posed diffusions in rectangular cylinders and propagate chaos}
\label{sec:systemes-bien-poses}

In this section, we check that the particle systems and the mean-field
limit systems are well-posed, and that the components of the
processes $(y^{i}_t),\ (x^{i}_t)$,
$(y^\alpha_t), \ (x^\alpha_t)$ take values in~$[0,1]$. Then
we prove that the particle systems propagate chaos towards the law of the limit systems~\eqref{eq:N-neuron-simple-conductance-limit}
and~\eqref{eq:N-neuron-sign-preserving-limit}.  

Our situation differs from the above mentioned Scheutzow's counterexamples~\cite{scheutzow}
in the fact that the interaction kernel is globally Lipschitz and the functions $F_\alpha$ are one-sided Lipschitz
(they are not only locally Lipschitz).
These features of the neuronal models under consideration protect one from non-uniqueness of solutions.

\subsection*{Well-posedness of the $N$-neuron models}
\begin{proposition} \label{prop:particle-systems-well-defined}
Under Hypothesis~\ref{hypo:coeff}
the systems~(\ref{eq:N-neuron-simple-conductance})
and~\eqref{eq:N-neuron-sign-preserving} have unique pathwise
solutions on all time interval $0\leq t\leq T$. In addition,
the components of the processes $(y^{i}_t)$, $(n^{i}_t)$, $(m^{i}_t)$, $(h^{i}_t)$ 
take values in~$[0,1]$. 
\end{proposition}

\begin{proof}
Observe that the coefficients of~(\ref{eq:N-neuron-simple-conductance})
and~(\ref{eq:N-neuron-sign-preserving}) are locally Lipschitz. This is 
obvious for the drift coefficients. In view of the assumption on the function 
$\chi$, the diffusion coefficients obviously are locally Lipschitz 
at all point $(v,x)$ (respectively, $(v,y)$) with $x$ or $y$ in 
$\er\setminus[0,1]$; this also holds true at all the other points since,
for all $\lambda_1>0$ and $\lambda_2>0$,
the value of $z$ such that $\lambda_1(1-z)+\lambda_2 z=0$ never belongs to $[0,1]$, so that
the arguments of the square roots in the diffusion coefficients are strictly positive when
$x$ (respectively, $y$) belongs to $[0,1]$.

It results from the preceding observation that  solutions to
~(\ref{eq:N-neuron-simple-conductance})
and~(\ref{eq:N-neuron-sign-preserving})
exist and are pathwise unique up to their respective explosion times:
see, e.g.,  Protter~\cite[Chap.V, Sec. 7, Thm.38]{protter}. Set
$$ \xi^{i}_t := (V^{i}_t,y^i_t,\bx^{i}_t)~~\text{or}~~
(V^{i}_t,J^{i\gamma}_t,y^i_t, 
\bx^{i}_t),\,\gamma\in\PP . $$

Using the one-sided Lipschitz  condition~(\ref{cond:one-sided-lipschitz}) and 
It\^o's formula, it is an easy exercise to get that 
$\EE|\xi^{i}_T|^2$ is finite for all $T>0$, from which it
readily follows that
$\EE\sup_{0\leq t\leq T}|\xi^{i}_t|^2$ is finite for all $T>0$.
Therefore the explosion times
of~(\ref{eq:N-neuron-simple-conductance})
and~(\ref{eq:N-neuron-sign-preserving}) are infinite a.s.

Let us now check that the coordinates of $(y^{i}_t,n^{i}_t,m^{i}_t, h^{i}_t)$ 
take values in $[0,1]$.  Their dynamics are of the type
\begin{equation}\label{eq:ionchanelsbis}
dx_t = \left( \rho_x(V_t)(1-x_t)  - \zeta_x(V_t )x_t\right) dt 
+ \sqrt{|\rho_x(V_t)(1-x_t) + \zeta_x(V_t )x_t|} \;\chi(x_t) dW^x_t 
\end{equation}
for $x= n, m, h$, and 
\begin{equation}\label{eq:ionchanelsy}
dy^{\alpha}_t =  \left(a^\alpha_rS_\alpha(V_t) 
(1-y^{\alpha}_t)  - a^\alpha_dy^{\alpha}_t \right) dt 
+ \sqrt{ \left| a^\alpha_r S_\alpha(V_t) (1-y^{\alpha}_t)  
	+ a^\alpha_dy^{\alpha}_t \right| } \; \chi(y^{\alpha}_t) 
dW^{\alpha,y}_t, 
\end{equation}
where $V_t$ is some real-valued continuous process. The hypothesis~\ref{hypo:H}-(ii) is  satisfied by the drift
coefficients of~(\ref{eq:ionchanelsbis}) and~(\ref{eq:ionchanelsy}): 
\begin{align*}
&(x,v) \mapsto b_x(x,v) := \rho_x(v)(1-x) - \zeta_x(v)x\\
\mbox{and}\qquad & (y,v) \mapsto b_y(y,v) := a_r ^{\alpha}S_\alpha(v) (1-y)  - a_d^{\alpha} y.
\end{align*}
The desired result follows, by applying Proposition~\ref{prop:main}.
\end{proof}

\begin{remark}
The preceding discussion shows that one can get rid of the absolute 
values in the diffusion coefficients of all the models in 
Section~\ref{sec:the-models}.
\end{remark}

\subsection*{Well-posedness of the mean-field limit models}~

For the next proposition we slightly reinforce the hypotheses on the functions $\rho_x$
and $\zeta_x$. The boundedness from below by strictly positive constant
is justified from a biological point of view (see the discussion in \cite[Sec.2.1,p.5]{baladron-al}).

\begin{proposition} \label{prop:EDS-non-lineaire}
Under Hypothesis~\ref{hypo:coeff} and the coercivity condition
\begin{align}\label{hypo:coercivity}
\exists \nu>0,~\forall v\in\er,~\rho_x(v)\wedge\zeta_x(v)\geq\nu>0,
\end{align}
the systems~\eqref{eq:N-neuron-simple-conductance-limit}
and~\eqref{eq:N-neuron-sign-preserving-limit}, complemented with \eqref{eq:N-neuron-limit-w} or \eqref{eq:N-neuron-limit-x},
have unique pathwise solutions on all time interval~$[0,T]$.
In addition, all the components of the process 
$(y^{\alpha}_t, n^\alpha_t,m^\alpha_t,h^\alpha_t)$ take values in~$[0,1]$.
\end{proposition}
\begin{proof}
Existence and pathwise uniqueness are obtained by slightly extending
the fixed point method developed by Sznitman~\cite{sznitman} for particle 
systems with bounded Lipschitz coefficients. We essentially combine
arguments already available in the literature 
(e.g. see~\cite{L-S} and references therein)
and therefore only
emphasize the additional minor arguments
required by the above neuron models. As exactly the same arguments can be used to
treat Equations~(\ref{eq:N-neuron-simple-conductance-limit}) and
~(\ref{eq:N-neuron-sign-preserving-limit}), we limit ourselves to consider the second one. 

We start with the following observation. 
Given the Brownian motions $(B^{\alpha\gamma}, \alpha,\gamma\in\PP )$ and the constant $\bar{J}^{\alpha\gamma}$, the  processes $(J^{\alpha\gamma}_t,\alpha,\gamma \in \PP )$  are unique pathwise solutions according to the Yamada and Watanabe Theorem (see e.g., Karatzas and Shreve~\cite[Chap.5, Thm.2.13]{K-S}). Let $\varphi(t)$ be an arbitrary
continuous function. Consider the system obtained by substituting 
$\varphi(t)$ for $\EE[y_t^{\gamma}]$ 
into~(\ref{eq:N-neuron-sign-preserving-limit}). Similar arguments as in
the proof of Proposition~\ref{prop:particle-systems-well-defined}
show that this new system has a unique pathwise solution. 

Now, we denote by $\ell=\bar{p}^2+6\bar{p}$  the dimension of the state space of
the process 
$$(V^\alpha_t,J^{\alpha\gamma}_t,y^{\alpha}_t, w^\alpha_t,(n^\alpha_t,m^\alpha_t,h^\alpha_t);
~\alpha,\gamma\in\PP ).$$
\begin{remark}\label{rmrk:FN+HH} 
Notice that we have lumped together the two models we are focusing on, i.e. Fitzhugh-Nagumo and Hodgkin-Huxley, since the mathematical tools for handling them are identical.
\end{remark}

\noindent
Let
$$(r_t) :=(v^\alpha_t,j^{\alpha\gamma}_t,\psi^\alpha_t,w^\alpha_t, z^\alpha_t;
~\alpha,\gamma\in\PP )$$
be the canonical process of $\mathcal{C}(0,T;\er^\ell)$.
Let $\mathcal{C}_T $ be the subspace of $\mathcal{C}(0,T;\er^\ell)$ 
consisting of the paths of the canonical process such that 
$(\psi^\alpha_t,z^\alpha_t)$ takes values in $[0,1]^4$ for all $t$, $\alpha\in\PP $.

Equip the space $\mathcal{M}_T$ of probability measures on 
$\mathcal{C}_T$ with the standard Wasserstein(2) metric:
$$ \Ww_T(\pi_1,\pi_2)
 := \left\{\min_{\mu\in\Lambda(\pi_1,\pi_2)} \int_{C_T\;} \sup_{0\leq s\leq T}
\left(
\left|r^1_s-r^2_s\right|^2 \right)
~\mu(dr^1,dr^2) \right\}^{1/2},$$ 
where $\Lambda(\pi_1,\pi_2)$ denotes the set of all coupling measures of
$\pi_1$ and $\pi_2$.

Given $\pi$ in $\mathcal{M}_T$, set
\begin{align*}
\overline{y}^{\mes,\gamma}_t & := \EE_\pi [y_0^{\gamma}] \exp\left\{ - a^\gamma_d t -  a_r^\gamma\int_0^t  
\EE_\pi[S_\gamma(v^\gamma_\theta)] d\theta \right\}\\
& \qquad + \int_0^t a_r^\gamma \EE_\pi[S_\gamma(v^\gamma_s)]\exp\left\{ - a^\gamma_d (t-s)  -  a_r^\gamma\int_s^t  
\EE_\pi[S_\gamma(v^\gamma_\theta)] d\theta \right\} ds,  \ \mbox{for all }\gamma\in\PP .
\end{align*}

Let us construct a contraction map $\Phi$ on $\mathcal{M}_T$ as follows.
For all $\pi$ in $\mathcal{M}_T$, $\Phi(\pi)$ is the probability law of the process 
$$(R^\mes_t) :=(V^{\mes,\alpha}_t,J^{\alpha\gamma}_t,y^{\mes,\alpha}_t,w^{\mes,\alpha}_t,
x^{\mes,\alpha}_t;~\alpha,\gamma\in\PP ;~x=n,m,h) $$
solution to
\begin{equation*} 
\begin{split}
dV^{\mes,\alpha}_t &= F_\alpha(t, V^{\mes,\alpha}_t , \bx^{\mes,\alpha}_t) dt   
- \sum_{ \gamma\in\PP } (V^{\mes,\alpha}_t 
- \overline{V}^{\alpha \gamma}) 
 J_t^{\alpha\gamma} \overline{y}^{\mes,\gamma}_t   dt   
+ \sigma_\alpha dW^{\alpha_t,V}, \\
d J^{\alpha\gamma}_t &= \theta_{\alpha\gamma}
\left( \bar{J}^{\alpha\gamma} 
- J^{\alpha\gamma}_t \right) dt + \sigma^{J}_{\alpha\gamma}  
\sqrt{J^{\alpha\gamma}_t} dB^{\alpha\gamma}_t,  \\
dy^{\mes,\alpha}_t &=  \left(a^\alpha_r S_\alpha(V^{\mes,\alpha}_t)(1-y^{\mes,\alpha}_t)  - a^\alpha_d y^{\mes,\alpha}_t \right) dt 
+ \sqrt{ \left| a^\alpha_r S_\alpha(V^{\mes,\alpha}_t)(1-y^{\mes,\alpha}_t)  
+ a^\alpha_d y^{\mes,\alpha}_t \right| } \;  \chi(y^{\mes,\alpha}_t)  dW^{\alpha,y}_t\\
dw^{\mes,\alpha}_t &= c_\alpha(V^{\mes,\alpha}_t + a_\alpha - b_\alpha 
w^{\mes,\alpha}_t) dt, \\
dx^{\mes,\alpha}_t  &= \left(  \rho_x(V^{\mes,\alpha}_t)(1-x^{\mes,\alpha}_t) 
- \zeta_x(V^{\mes,\alpha}_t )x^{\mes,\alpha}_t\right) dt \\
&\qquad + \sqrt{| \rho_x(V^{\mes,\alpha}_t)(1-x^{\mes,\alpha}_t) 
+ \zeta_x(V^{\mes,\alpha}_t )x^{\mes,\alpha}_t|} \;\chi(x^{\mes,\alpha}_t) 
dW^{\alpha,x}_t
\end{split}
\end{equation*}
with $x=n,m,h$, $\bx=(w)$ or $\bx=(x)$.  
Notice that the probability law of the process $(R^\mes_t)$ is supported in
$\mathcal{C}_T$.

Set $ \Delta R_t := R^{(\pi_1)}_t - R^{(\pi_2)}_t$, and
apply It{\^o}'s formula to $ (\Delta R_t)^2$.
In order to deduce that there exists a positive constant $K_T$ uniform w.r.t. $\pi_1$ and $\pi_2$ such that
\begin{equation} \label{ineq:delta-Z}
\forall 0\leq t\leq T,~~\EE|\Delta R_t|^2 \leq K_T\left(
\int_0^t \EE|\Delta R_s|^2 ds + \int_0^t \Ww_s(\pi_1,\pi_2) ds \right),
\end{equation}
it suffices to use classical arguments, plus the following ingredients:
\begin{itemize}
\item The one-sided Lipschitz  condition~(\ref{cond:one-sided-lipschitz});
\item The fact that
$\overline{y}^{\mes,\alpha}_t$ is uniformly bounded w.r.t. $\pi$
in $\mathcal{M}_T$ and $t$ in $[0,T]$;
\item The additional coercivity condition 
\eqref{hypo:coercivity} implies
$$\forall x\in[0,1],~\forall v\in\er, ~ \rho_x(v) (1-x) + \zeta_x(v) x 
\geq \nu >0. $$
As $(x^{(\pi_1),\alpha}_t)$ and $(x^{(\pi_2),\alpha}_t)$ take values in 
$[0,1]$, all the terms of the type
\begin{multline}\label{bound:quadraticDiff}
\int_0^t \left( 
\sqrt{| \rho_x(V^{(\pi_1),\alpha}_s)(1-x^{(\pi_1),\alpha}_s) 
+ \zeta_x(V^{(\pi_1),\alpha}_s)x^{(\pi_1),\alpha}_s|\;} \chi(x^{(\pi_1),\alpha}_s)
\right. \\
\left.
- \sqrt{| \rho_x(V^{(\pi_2),\alpha}_s)(1-x^{(\pi_2),\alpha}_s) 
+ \zeta_x(V^{(\pi_2),\alpha}_s)x^{(\pi_2),\alpha}_s|\;}\chi(x^{(\pi_2),\alpha}_s)
\;\right)^2  ds 
\end{multline}
are bounded from above by
$$ K_T\int_0^t  | x^{(\pi_1),\alpha}_s-x^{(\pi_2),\alpha}_s |^2 ds
+ K_T\int_0^t  | V^{(\pi_1),\alpha}_s-V^{(\pi_2),\alpha}_s |^2 ds,$$
(the same remarks apply to the diffusion coefficients of $(y^{(\pi_1),\alpha}_t)$ and $(y^{(\pi_2),\alpha}_t)$);
\item The existence of a positive $K_T$ uniform w.r.t. $\pi_1$ and $\pi_2$ such
that, for all $\alpha\in\PP $, 
$$\sup_{0\leq t\leq T}\left| \overline{y}^{(\pi_1),\alpha}_t
- \overline{y}^{(\pi_2),\alpha}_t \right|^2
\leq K_T \int_0^T \Ww^2_s(\pi_1,\pi_2)ds. $$
\end{itemize}
Classical arguments allow one to deduce from~(\ref{ineq:delta-Z}) that,
for some possibly new positive constant $K_T$,
$$ \Ww^2_T(\Phi(\pi_1),\Phi(\pi_2))
\leq K_T \int_0^T \Ww^2_s(\pi_1,\pi_2)ds. $$
One then obtains that $\Phi$ is a contraction map (see~Sznitman~\cite{sznitman}), 
from which the desired existence and pathwise uniqueness result follows 
for \eqref{eq:N-neuron-sign-preserving-limit}.

It remains to again use Proposition~\ref{prop:main} to get the last part in the statement.
\end{proof}

\subsection*{Convergence}

In this part, we analyze the convergence of the $N$-neurons system given in \eqref{eq:N-neuron-sign-preserving} to the mean-field limit described in \eqref{eq:N-neuron-sign-preserving-limit}. The convergence of the model \eqref{eq:N-neuron-simple-conductance} to the solution of \eqref{eq:N-neuron-simple-conductance-limit} results from a straightforward adaptation of the following  proposition and of its proof. 
Moreover, as in the proof of Proposition \ref{prop:EDS-non-lineaire} we use again Remark~\ref{rmrk:FN+HH} to shorten the presentation.

Let $(R_t) = (R^\alpha_t,\alpha\in\PP )  = (V^\alpha_t,(J^{\alpha\gamma}_t,\gamma \in \PP ), y^\alpha_t,  w^\alpha_t,n^\alpha_t, m^\alpha_t, h^\alpha_t;\alpha\in\PP )$ denote the solution of \eqref{eq:N-neuron-sign-preserving-limit}, with law $\mathbb{P}$ on $C_T$.  
Let $(R^{i,N}_t,i=1,\ldots,N)= (V^i_t,(J^{i\gamma}_t, \gamma \in \PP ),y^i_t,  w^i_t,n^i_t, m^i_t,h^i_t;i=1,\ldots,N)$, the solution of the N-neurons system \eqref{eq:N-neuron-sign-preserving}.  
Considering the family of Brownian motions in~ \eqref{eq:N-neuron-sign-preserving}, and the set of i.i.d  initial random variables $(V^i_0,(J^{i\gamma}_0, \gamma \in \PP ), y^i_0, w^i_0,n^i_0,m^i_0,h^i_0;i=1,\ldots,N)$, 
we introduce a coupling between the $(R^{i,N}_t,i=1,\ldots,N)$ and a set of $N$ processes $({\tR}^i_t,i=1,\ldots,N)$ such that for all $\alpha \in \PP $, all $N_\alpha$ indices $i$ such that $p(i)=\alpha$ are such that $({\tR}^i_t)$ are independent copies of $(R^\alpha_t)$. More precisely, for each $i=1,\ldots,N$, such that $p(i)=\alpha$,  $({\tR}^{i}_t)= (\widetilde{V}^i_t,(\widetilde{J}^{i\gamma}_t, \gamma \in \PP ),\widetilde{y}^i_t,  \widetilde{w}^i_t,\widetilde{n}^i_t,\widetilde{m}^i_t,\widetilde{h}^i_t)$ is the solution of 
\begin{align*}
\left\{
\begin{aligned}
& \mbox{ for }\tbx = (\widetilde{w}) \mbox{ or } (\widetilde{n},\widetilde{m},\widetilde{h}), \\
&d\widetilde{V}^{i}_t = F_\alpha(t, \widetilde{V}^{i}_t , \tbx^{i}_t) dt   
- \sum_{ \gamma\in\PP }  (\widetilde{V}^{i}_t - \overline{V}^{\alpha \gamma})
\widetilde{J}_t^{i\gamma} \EE[y_t^\gamma] dt  + \sigma_\alpha dW^{i,V}_t, \\
&(\widetilde{J}^{i \gamma}_t, t\geq 0) =  (J^{i \gamma}_t, t\geq 
0), \text{ for all }\gamma \in \PP, \\
& d\widetilde{y}^{i}_t =  \left(a^\alpha_rS_\alpha(\widetilde{V}^{i}_t) 
(1-\widetilde{y}^{i}_t)  - a^\alpha_d\widetilde{y}^{i}_t \right) dt 
+ \sqrt{ \left|a^\alpha_rS_\alpha(\widetilde{V}^{i}_t) (1-\widetilde{y}^{i}_t) 
+ a^\alpha_d\widetilde{y}^{i}_t \right| } \; \chi(\widetilde{y}^{i}_t) 
dW^{i,y}_t, 
\end{aligned}
\right.
\end{align*}
coupled with 
\begin{align*} 
d\widetilde{w}^{i}_t = c_\alpha(\widetilde{V}^{i}_t + a_\alpha - b_\alpha \widetilde{w}^{i}_t) dt
\end{align*}
or
\begin{align*}
d\widetilde{x}^{i}_t  &= \left( \rho_{x}(\widetilde{V}^{i}_t)
(1-\widetilde{x}^{i}_t) 
- \zeta_{x}(\widetilde{V}^{i}_t)\widetilde{x}^{i}_t\right) dt + \sqrt{| \rho_{x}(\widetilde{V}^{i}_t)(1-\widetilde{x}^{i}_t) 
+ \zeta_{x}(\widetilde{V}^{i}_t )
\widetilde{x}^{i}_t|} \;\chi(\widetilde{x}^{i}_t) dW^{i,x}_t
\end{align*}
for $\widetilde{x}=\widetilde{n},\widetilde{m},\widetilde{h}$,  and starting at $(V^i_0,(J^{i\gamma}_0, \gamma \in \PP ), y^i_0, w^i_0,n^i_0,m^i_0,h^i_0)$. 

Under the hypotheses of Proposition~\ref{prop:EDS-non-lineaire}
we have the following propagation of chaos result:

\begin{proposition}\label{prop:propagationChaos}
Assume that for all $\gamma\in \PP $, the proportion $N_\gamma/N$ of neurons in population $\gamma$ is a nonzero constant independent of $N$ and denoted:
\begin{align}\label{hypo:proportion}
c_\gamma = \dfrac{N_\gamma}{N}.
\end{align}
Then there exists a constant $C>0$ such that, for all $N = \sum_{\gamma\in\PP } N_\gamma$ satisfying the Assumption \eqref{hypo:proportion}, for all set of $\bar{p}$ indices $(i_\alpha, \alpha\in  \PP )$ among $\{1,\ldots,N\}$ with  $p(i_\alpha) = \alpha$,   
the vector process $(R^{i_\alpha,N} - {\tR}^{i_\alpha})$, with one component in each population, satisfies
\begin{equation} \label{ineq:vitesse-cv-propag-chaos}
\sqrt{N}\; \EE\left[ \sup_{0\leq t\leq T} \sum_{\alpha \in \PP } 
|R^{i_\alpha,N}_t-{\tR}^{i_\alpha}_t|^2
\right] \leq C.
\end{equation}
\end{proposition}
The above $L^2(\Omega)$-estimate obviously implies that the law of any subsystem of size $k$ $$\left((R^{1_\alpha,N}_t,\alpha\in\PP ),\ldots,(R^{k_\alpha,N}_t,\alpha\in\PP )\right)$$  
with $p(i_\alpha) = \alpha$, converges to the law $\mathbb{P}^{\otimes k}$ when the $N_\alpha$s tend to $\infty$. In other words, the reordered system 
$$\left((R^{i_\alpha,N}_t,\alpha\in\PP ), p(i_\alpha) = \alpha, i_\alpha\in\{1,\ldots,N\}\right)$$
is $\mathbb{P}$-chaotic. 

\begin{proof}[Proof of Proposition \ref{prop:propagationChaos}. ]
~\paragraph{A short  discussion on our methodology. }
We only
present the main ideas of the proof which follows and adapts Sznitman 
\cite{sznitman} or M{\'e}leard ~\cite{meleard}. 
To help the reader follow the lengthy calculations, we start with an explanation of  the main 
differences between our problem where some of the coefficients of our stochastic differential equations are not globally  Lipschitz continuous and the classical Lipschitz coefficients framework. 
In a nutshell, we are dealing with a particle system of the  generic form 
$$ dX^i_t = f(X^i_t,\tfrac{1}{N}\sum_{j=1}^N\phi(X^j_t))dt 
+ \sigma(X^i_t)dW^i_t, $$
where the Brownian motions $W^i$ are independent,
and the functions $\phi$, $f$ and $\sigma$ are such that one gets existence 
and strong uniqueness of a
solution with finite moments, as well as the existence and 
strong uniqueness of the mean field limit
$$ dX_t = f(X_t,\EE\phi(X_t))dt + \sigma(X_t)dW^1_t. $$
Now, let $(\widetilde{X}^i_t)$ be independent copies 
of $(X_t)$ driven by the Brownian motions $W^i$. Under strong 
enough Lipschitz hypotheses on $\phi$, $f$ and $\sigma$, one typically 
obtains, for some $C>0$ and all $0\leq t\leq T$,
$$ \EE |X^i_t-\widetilde{X}^i_t|^2
\leq C \int_0^t \EE |X^i_\theta - \widetilde{X}^i_\theta |^2 d\theta
+ C \int_0^t \EE |\EE\phi(\widetilde{X}^1_\theta) - 
\tfrac{1}{N}\sum_{j=1}^N\phi(\widetilde{X}^j_\theta)|^2~d\theta. $$
Using independence arguments one readily gets that
$$\EE\big|\EE\phi(\widetilde{X}^1_\theta) - 
\tfrac{1}{N}\sum_{j=1}^N\phi(\widetilde{X}^j_\theta)\big|^2
\leq \dfrac{C}{N}.$$
Using Gronwall's lemma, one deduces that
$$ \EE |X^i_t-\widetilde{X}^i_t|^2 \leq \dfrac{C}{N}. $$
This method fails when one of  $f$, $\sigma$ or  $\phi$ is not 
globally Lipschitz. 

In our case the drift $f$ is not globally Lipschitz, but of the form  (see equations  \eqref{eq:drift-B}, \eqref{eq:drift-b}, \eqref{eq:kernel-k})
\[f(t,v,\bx,j,\tfrac{1}{N}\sum_{i=1}^N y^i) = F_\alpha(t,v,\bx) -  j (v- \overline{V}^{\alpha \gamma})  (\tfrac{1}{N}\sum_{i=1}^N y^i).\]
The fact that the first part $F_\alpha$ of the drift is only one-sided Lipschitz is easy to overcome.  To handle the second part $-  j (v- \overline{V}^{\alpha \gamma})  (\tfrac{1}{N}\sum_{i=1}^N y^i)$ we make use of the following three properties: 
\begin{itemize}
\item the processes $J_t^{\alpha\gamma}$ are positive,
\item the processes $y^i_t$ belong to $[0,1]$,
\item in the dynamics of $V_t$, the term  $-  j (v- \overline{V}^{\alpha \gamma})  (\tfrac{1}{N}\sum_{i=1}^N y^i)$ acts as a mean reverting term, stabilizing the moments of  $V_t$. 
\end{itemize}
Notice that because in our case the function $f$ is not globally Lipschitz,  the convergence rate for \newline
$\EE\left[ \sup_{0\leq t\leq T} \sum_{\alpha \in \PP }
|R^{i_\alpha,N}_t-{\tR}^{i_\alpha}_t|^2\right]$ is of the order of  $1/\sqrt{N}$,
as indicated by the inequality \eqref{ineq:vitesse-cv-propag-chaos}, instead of $1/N$ in the Lipschitz case.

\paragraph{Details of our  proof. } We now turn to the proof of Inequality~(\ref{ineq:vitesse-cv-propag-chaos}).

By the symmetry of the coupled systems, we can fix the index set  $(1_\alpha, \alpha\in  \PP )$ and rewrite the SDEs \eqref{eq:N-neuron-sign-preserving} and \eqref{eq:N-neuron-sign-preserving-limit} in the following condensed form:  for all $\alpha\in \PP $,
\begin{align}\label{eq:R-eds}
R^{1_\alpha,N}_t -{\tR}^{1_\alpha}_t =  \int_0^t \left( \sigma(R^{1_\alpha,N}_s)  - \sigma({\tR}^{1_\alpha}_s)\right) d {\bW}^{1_\alpha}_s + \int_0^t \left( B[s,R^{1_\alpha,N}_s; \mu^N_s]  - B[s,{\tR}^{1_\alpha}_s;\mathbb{P}_s]\right) ds
\end{align}
where we have introduced the empirical measure $\mu^N_\cdot = 
\dfrac{1}{N}\displaystyle\sum_{j=1}^N \delta_{R^j_\cdot}$, the Brownian 
motion $(\bW^{1_\alpha}_t) = 
(W^{1_\alpha,V}_t, 
(B_t^{1_\alpha\gamma},\gamma \in \PP),W^{1_\alpha,y}_t,W^{1_\alpha,x}_t)$, 
and the time-marginal law $\mathbb{P}_s = \mathbb{P} \circ ({\tR}^{1_\alpha}_s,\alpha\in\PP )^{-1}$.

We denote by $r$ the canonical variable on $\er^\ell$,  that we decompose in the following set of $\bar{p}$ coordinates on $\er^{\bar{p}+6}$:  
$$r :=(r^\alpha, \alpha\in\PP ) = 
\left(v^\alpha,(j^{\alpha\gamma};~\gamma\in\PP ),y^\alpha,w^\alpha, x^\alpha;~\alpha\in\PP \right).$$
According to Remark \ref{rmrk:FN+HH}, the diffusion coefficient is 
defined as  
\begin{align*}
\sigma(r^\alpha)= \sigma(v^\alpha,(j^{\alpha\gamma},\gamma \in \PP ), y^\alpha, w^\alpha,x^\alpha)= \left(\begin{array}{l}
\sigma^\alpha\\
\left( \sigma^J_{\alpha\gamma}\sqrt{j^{\alpha\gamma}}, \gamma \in \PP  \right) \\
\sqrt{ \left| a^\alpha_r S_\alpha(v^{\alpha}) (1-y^{\alpha})  
	+ a^\alpha_dy^{\alpha}\right| } \; \chi(y^{\alpha}) \\
0 \\
\sqrt{\displaystyle | \rho_{x}(v^\alpha)(1-x^\alpha) 
+ \zeta_{x}(v^\alpha)x^\alpha|} \;\chi(x^\alpha) 
\end{array}\right)
\end{align*}
and is Lipschitz on the state subspace
of the process $(V^\alpha_t,y^\alpha_t, w^\alpha_t, x^\alpha_t)$.  
The drift coefficient is defined as 
\begin{align}\label{eq:drift-B}
B[t,r^\alpha;\mu]:= b(t,r^\alpha) + k[r^\alpha;\mu]
\end{align}
where
\begin{align}\label{eq:drift-b}
b(t,r^\alpha)= b(v^\alpha,(j^{\alpha\gamma},\gamma \in \PP ), y^\alpha, w^\alpha,x^\alpha)= \left(\begin{array}{l}
F_\alpha(t, v^\alpha , \bx^\alpha)\\  
\left( \theta_{\alpha\gamma}
\left( \bar{J}^{\alpha\gamma} 
- j^{\alpha\gamma} \right),\gamma\in\PP \right)\\
\left(a^\alpha_r S_\alpha(v^{\alpha}) 
(1-y^{\alpha})  - a^\alpha_dy^{\alpha} \right)\\
c_\alpha(v^\alpha + a_\alpha 
- b_\alpha w^\alpha)\\
\left(\rho_{x}(v^\alpha)(1-x^\alpha) 
- \zeta_{x}(v^\alpha )x^\alpha\right)
\end{array}\right)
\end{align}
is one-sided Lipschitz in the sense of \eqref{cond:one-sided-lipschitz} in Hypothesis \ref{hypo:coeff}-(iii), and $k$ is defined as follows. For any probability measure $\mu$ on $\er^{l}$, 
\begin{align}\label{eq:kernel-k}
k[r^\alpha;\mu]=  
\left(\begin{array}{l}
- \displaystyle\int_{\er^\ell} \displaystyle \sum_{\gamma\in\PP } (v^\alpha - \overline{V}^{\alpha\gamma}) j^{\alpha\gamma} \frac{1}{c_\gamma} \indi{\{\eta=\gamma\}}  y^{\eta} \mu(d(r^\eta;~\eta\in\PP ))
\\  
0\\
0\\
0\\
0
\end{array}\right).
\end{align}

\begin{remark}
Notice that the characteristic function $\indi{\{\eta=\gamma\}}$ is  
unnecessary in the above definition but, combined  with the 
hypothesis  \eqref{hypo:proportion} that fixes the constants $\{c_\gamma;\gamma\in\PP \}$, 
 it has the advantage of matching the notations in equations 
 \eqref{eq:N-neuron-simple-conductance} and 
 \eqref{eq:N-neuron-sign-preserving}, which helps identifying the 
 interaction kernel in the limit equations. 
The measure $\mu(d(r^\eta;~\eta\in\PP ))$ is on $\er^\ell$  
whose state coordinates are here labeled in $\PP $. 
\end{remark}

In all the sequel $C$ is a positive constant which may  change from line 
to line and is independent of $N$ and $0\leq t\leq T$,  but  may depend of $T$.

\paragraph{Step 1. We prove that the processes $V^i_t$ have bounded 
moments of any positive order.  }
We start with reminding the reader that the CIR processes 
$(J^{i\gamma})$ have bounded 
moments of any positive order when their initial conditions enjoy the 
same property (see e.g. Lemma 2.1 in Alfonsi \cite{alfonsi}).  In view 
of the 
hypotheses~\ref{hypo:coeff}-(i) and (iv), one can show that
the same is true for the processes $(V^{i})$ and $(\widetilde{V}^{i})$
by using the following argument.
Apply the It{\^o} 
formula to $(V^{i})^{2q}$, $q\geq 1$, till time $\tau_n = \inf\{t\geq 
0; |V^{i}_t|\geq n\}$, and take expectations to get 
\begin{align*}
\EE[(V^{i}_{t\wedge \tau_n})^{2q}] & = \EE[(V^{i}_0)^{2q}] + 2q 
\int_0^{t} \EE[\indi{\{s\leq \tau_n\}}(V^{i}_s)^{2q-1} F_\alpha(t, 
V^{i}_s , 
\bx^{i}_s)] ds \\  
&\quad\quad - 2q \int_0^{t}\EE\Big[\indi{\{s\leq \tau_n\}}\sum_{ 
\gamma \in \PP }  (V^{i}_s)^{2q-1}(V^{i}_s - \overline{V}^{\alpha 
\gamma})
\dfrac{{J}^{i\gamma}_s}{N_\gamma} 
\left(\sum_{j=1}^N 
\indi{\{p(j)=\gamma\}}y^{j}_s\right) \Big]ds \\
&\quad\quad 
+ q (2q-1)\int_0^{t}\EE\left[\indi{\{s\leq \tau_n\}}(V^{i}_s)^{2q-2} 
(\sigma_\alpha)^2\right] ds.
\end{align*}
We then observe that 
\[- 2q \int_0^{t}\EE\Big[\indi{\{s\leq \tau_n\}}\sum_{ \gamma \in \PP 
}  (V^{i}_s)^{2q}
\dfrac{{J}^{i\gamma}_s}{N_\gamma} 
\Big(\sum_{j=1}^N 
\indi{\{p(j)=\gamma\}}y^{j}_s\Big) \Big]ds \]
is negative  and can be ignored. It then remains to use the 
hypotheses~\ref{hypo:coeff} and classical arguments to deduce that
$\EE\left[(V^{i}_t)^{n}\right]$ is finite for all $n\geq 1$.

\paragraph{Step 2. A first bound for the random variables  $|R^{1_\alpha,N}_t 
-{\tR}^{1_\alpha}_t|^2$ and $(\tfrac{1}{N} \sum_{i=1}^N (y^i_t - 
\ty^i_t))^2$. } 
Because of the polynomial  form of the non Lipschitz part of the drift, it is not a good idea to introduce the expectation too early in the calculation of the bound for $|R^{1_\alpha,N}_t -{\tR}^{1_\alpha}_t|^2$ or $(\tfrac{1}{N} \sum_{i=1}^N (y^i_t - \ty^i_t))^2$, since calculations with expectation lead to terms of the type $\EE[(R^{1_\alpha,N}_t -{\tR}^{1_\alpha}_t)^2 H]$, where $H$ is an unbounded random variable correlated with $R^{1_\alpha,N}_t$. We therefore postpone taking expectations to Step 3. 

Apply It{\^o}'s formula to $|R^{1_\alpha,N}_t 
-{\tR}^{1_\alpha}_t|^2$. We 
obtain
\begin{align*}
|R^{1_\alpha,N}_t -{\tR}^{1_\alpha}_t|^2 = & 
\; 2  \int_0^t \left( B[s,R^{1_\alpha,N}_s; \mu^N_s]  - 
B[s,{\tR}^{1_\alpha}_s;\mathbb{P}_s]\right) (R^{1_\alpha,N}_s 
-{\tR}^{1_\alpha}_s) ds \\
& + \int_0^t | \sigma(R^{1_\alpha,N}_s) - 
\sigma({\tR}^{1_\alpha}_s)|^2 ds + M^{1_\alpha,N}_t,
\end{align*}
where
\[
M^{1_\alpha,N}_t=2\int_0^t (R^{1_\alpha,N}_s -{\tR}^{1_\alpha}_s) \left( \sigma(R^{1_\alpha,N}_s)  - \sigma({\tR}^{1_\alpha}_s)\right) d {\bW}^{1_\alpha}_s
\]
is a martingale.
By It{\^o} isometry and the result in Step 1 above, 
$ \sup_{0\leq t\leq T}\EE |M^{1_\alpha,N}_t|^2 \leq C$.\\
Applying the  Lipschitz and one-sided-Lipschitz properties for $\sigma$ 
and $b$, we obtain  
\begin{align}\label{eq:couplage_aux}
\begin{aligned}
|R^{1_\alpha,N}_t -{\tR}^{1_\alpha}_t|^2 \leq & 
\; 2  \int_0^t \left( k[s,R^{1_\alpha,N}_s; \mu^N_s]  - 
k[s,{\tR}^{1_\alpha}_s;\mathbb{P}_s]\right) (R^{1_\alpha,N}_s 
-{\tR}^{1_\alpha}_s) ds \\
& + C \int_0^t |R^{1_\alpha,N}_s -{\tR}^{1_\alpha}_s|^2 ds 
+ M^{1_\alpha,N}_t.
\end{aligned}
\end{align}
Now, we are interested in $\left( k[s,R^{1_\alpha,N}_s; \mu^N_s]  - 
k[s,{\tR}^{1_\alpha}_s;\mathbb{P}_s]\right) (R^{1_\alpha,N}_s 
-{\tR}^{1_\alpha}_s) $. We introduce the empirical measure of the 
coupling system $\widetilde{\mu}^N_\cdot = \displaystyle 
\frac{1}{N}\sum_{i=1}^N
\delta_{{\tR}^i_\cdot}$ and write 
\begin{align}\label{eq:decompoCouplage}
\left( k[R^{1_\alpha,N}_s; \mu^N_s]  - k[{\tR}^{1_\alpha}_s;\mathbb{P}_s]\right)=\left( k[R^{1_\alpha,N}_s; \mu^N_s]  - k[{\tR}^{1_\alpha}_s;\widetilde{\mu}^N_s]\right) + 
\left( k[{\tR}^{1_\alpha}_s; \widetilde{\mu}^N_s]  - k[{\tR}^{1_\alpha}_s;\mathbb{P}_s]\right).
\end{align} 
We consider in turn the two terms in the right-hand side of \eqref{eq:decompoCouplage}.

First, from the definition of $k$ in \eqref{eq:kernel-k} we get 
\begin{align*}
& \left( k[R^{1_\alpha,N}_s; \mu^N_s]  - k[{\tR}^{1_\alpha}_s;\widetilde{\mu}^N_s]\right)(R^{1_\alpha,N}_s -{\tR}^{1_\alpha}_s)\\
& = \Big( \frac{1}{N} \sum_{j=1}^N
\sum_{\gamma\in\PP } \Big[ 
-(V^{1_\alpha}_s - \overline{V}^{\alpha\gamma}) J^{1_{\alpha}\gamma}_s 
\frac{1}{c_\gamma} \indi{\{p(j) = \gamma\}}  y^j_s 
+ (\widetilde{V}^{1_\alpha}_s - \overline{V}^{\alpha\gamma}) 
{J}^{1_{\alpha}\gamma}_s  \frac{1}{c_\gamma}\indi{\{p(j) = \gamma\}}  
\widetilde{y}^j_s  
\Big]\Big)(V^{1_\alpha}_s - \tV^{1_\alpha}_s)\\
& =   - (V^{1_\alpha}_s - \tV^{1_\alpha}_s)^2 
 \Big(\sum_{\gamma\in\PP } 
 J^{1_{\alpha}\gamma}_s \frac{1}{N} \sum_{j=1}^N
\frac{1}{c_\gamma} \indi{\{p(j) = \gamma\}}  y^j_s \Big) \\
& \qquad + 
(V^{1_\alpha}_s - \tV^{1_\alpha}_s)\Big( \sum_{\gamma\in\PP } {J}^{1_{\alpha}\gamma}_s (\widetilde{V}^{1_\alpha}_s - \overline{V}^{\alpha\gamma})
 \frac{1}{N} \sum_{j=1}^N
 \frac{1}{c_\gamma}\indi{\{p(j) = \gamma\}}  (\widetilde{y}^j_s - 
 y^j_s) \Big). 
\end{align*}
Since the $(J^{1_{\alpha}\gamma}_t)$  and  the $(y^i_t,i=1,\ldots,N)$
are positive, the first term in the right-hand side is negative.  
We bound the second term by using Young's inequality: 
\begin{align} \label{ineq:delta-k}
& \left(k[R^{1_\alpha,N}_s; \mu^N_s]  - 
k[{\tR}^{1_\alpha}_s;\widetilde{\mu}^N_s]\right)(R^{1_\alpha,N}_s 
-{\tR}^{1_\alpha}_s) \nonumber \\
& \leq \frac{1}{2} (V^{1_\alpha}_s - \tV^{1_\alpha}_s)^2 + 
\frac{1}{2} \Big( \sum_{\gamma\in\PP } {J}^{1_{\alpha}\gamma}_s 
(\widetilde{V}^{1_\alpha}_s - \overline{V}^{\alpha\gamma})
 \frac{1}{N} \sum_{j=1}^N
 \frac{1}{c_\gamma}\indi{\{p(j) = \gamma\}}  (\widetilde{y}^j_s - 
 y^j_s) \Big)^2. 
\end{align}

Next we consider the second contribution coming from the right-hand side 
of \eqref{eq:decompoCouplage}. By Young's inequality
\begin{align}\label{ineq:decompoCouplage2}
 \left(k[{\tR}^{1_\alpha}_s;\widetilde{\mu}^N_s] - k[{\tR}^{1_\alpha}_s;\mathbb{P}_s] \right)(R^{1_\alpha,N}_s 
-{\tR}^{1_\alpha}_s) 
 \leq \frac{1}{2} |R^{1_\alpha,N}_s 
-{\tR}^{1_\alpha}_s|^2 + \frac{1}{2}  (\zeta^{1_\alpha}_s)^2,
\end{align}
where
\[\zeta^{1_\alpha}_s:= k[{\tR}^{1_\alpha}_s; \widetilde{\mu}^N_s]  - 
k[{\tR}^{1_\alpha}_s;\mathbb{P}_s]\] 
is such that $\sup_{0\leq t\leq T}\EE| \zeta^{1_\alpha}_s|^2 \leq \tfrac{C}{N}$.  Indeed, as the $({\tR}^{i})$ 
are  i.i.d.  with law $\mathbb{P}$,  
$k[{\tR}^{1_\alpha}_s;\mathbb{P}_s]$ is the conditional expectation
\begin{align*}
k[{\tR}^{1_\alpha}_s;\mathbb{P}_s] &  
= \EE \left[k({\tR}^{1_\alpha}_s,{\tR}^j_s) \Big/ \sigma( {\tR}^{1_\alpha}_u; u\leq s) \right] 
\end{align*}
for any $j \neq 1_\alpha$, where we have set  $k(r^\alpha,r^\eta)= 
\sum_{\gamma\in\PP } (v^\alpha - \overline{V}^{\alpha\gamma}) 
j^{\alpha\gamma} \frac{1}{c_\gamma} \indi{\{\eta = \gamma\}}  y^\eta$, 
and the  symbol $\sigma$ stands for sigma algebra (which must not be 
confused with the above diffusion coefficient). Thus 
\begin{align*}
\EE | \zeta^{1_\alpha}_s |^2 
& = \EE \Big| \frac{1}{N} \sum_{j= 1}^N  k({\tR}^{1_\alpha}_s, {\tR}^{j}_s)  - \EE
\left[k({\tR}^{1_\alpha}_s,{\tR}^{j}_s) \Big/ \sigma( {\tR}^{1_\alpha}_u; u\leq s) \right] \Big| ^2 \\
& \leq 2 \EE \Big| \frac{1}{N} \sum_{j \neq 1_\alpha } k({\tR}^{1_\alpha}_s, {\tR}^{j}_s)  - \EE\left[k({\tR}^{1_\alpha}_s,{\tR}^{j}_s)\Big/ \sigma( {\tR}^{1_\alpha}_u; u\leq s) \right] \Big|^2 \\
& \quad + \frac{2}{N}\EE \Big|  k({\tR}^{1_\alpha}_s,{\tR}^{1_\alpha}_s) - \EE[k(r^\alpha,{\tR}^{1_\alpha}_s)]\Big| _{\{r^\alpha = {\tR}^{1_\alpha}_s\}} \Big|^2 \leq \frac{C}{N}. 
\end{align*}

Combine \eqref{eq:couplage_aux} with the last inequalities \eqref{ineq:delta-k},
\eqref{eq:decompoCouplage} and  \eqref{ineq:decompoCouplage2}, 
\begin{align*}
\begin{aligned}
|R^{1_\alpha,N}_t -{\tR}^{1_\alpha}_t|^2 \leq & 
 C\int_0^t |R^{1_\alpha,N}_s -{\tR}^{1_\alpha}_s|^2 ds \\
 & +  C\int_0^t \sum_{\gamma\in\PP } 
 (J^{1_{\alpha}\gamma}_s)^2 
(\widetilde{V}^{1_\alpha}_s - \overline{V}^{\alpha\gamma})^2
\Big(\frac{1}{N} \sum_{j=1}^N
  (\widetilde{y}^j_s - y^j_s) \Big)^2ds 
+ M^{1_\alpha,N}_t +  \frac{1}{2}  (\zeta^{1_\alpha}_t)^2.
\end{aligned}
\end{align*}
By Gronwall's lemma and integration by parts,
it comes
\begin{equation} \label{ineq:delta-R-recap}
|R^{1_\alpha,N}_t -{\tR}^{1_\alpha}_t|^2 
\leq C\int_0^t \sum_{\gamma\in\PP } 
 (J^{1_{\alpha}\gamma}_s)^2 
(\widetilde{V}^{1_\alpha}_s - \overline{V}^{\alpha\gamma})^2
\Big(\frac{1}{N} \sum_{j=1}^N
  (\widetilde{y}^j_s - y^j_s) \Big)^2~ds + Z^{1_\alpha}_t,
 \end{equation}
where for all $t\in[0,T]$, since $(M^{1_\alpha,N}_t)$ is a martingale, 
$$\EE [Z^{1_\alpha}_t]  = \EE\Big[M^{1_\alpha,N}_t +  \frac{1}{2}  (\zeta^{1_\alpha}_t )^2+\int_0^t C\exp(C(t-s)) \big(M^{1_\alpha,N}_s +  \frac{1}{2}  (\zeta^{1_\alpha}_s)^2\big) ds \Big]\leq C \sup_{0\leq t\leq T}\EE| \zeta^{1_\alpha}_s|^2 \leq \frac{C}{N}.$$

We now set
$$\dyb_t  = \frac{1}{N} \sum_{i=1}^N  (y^i_t - \ty^i_t). $$ 
Defining drift and diffusion for processes $y^i$ by 
\begin{align*}
b^\alpha_y(y,v)  &=a^\alpha_r S_\alpha(v) (1-y) - a^\alpha_d y\\
\sigma^\alpha_y(y,v)  &= \sqrt{ \left| a^\alpha_r S_\alpha(v) (1-y)  
	+ a^\alpha_d y\right| } \; \chi(y), 
\end{align*}
we have 
\begin{align*}
(\dyb_t)^2  = & \left( \int_0^t  \tfrac{1}{N}\sum_{i=1}^N  (b^\alpha_y(y_s^i,V_s^i)  - b^\alpha_y(\ty_s^i,\tV_s^i) ) ds + \int_0^t   \tfrac{1}{N}\sum_{i=1}^N  (\sigma^\alpha_y(y_s^i,V_s^i)  - \sigma^\alpha_y(\ty_s^i,\tV_s^i) ) dW_s^i \right)^2 \\
\leq & 2 \left( \int_0^t  \tfrac{1}{N}\sum_{i=1}^N  (b^\alpha_y(y_s^i,V_s^i)  - b^\alpha_y(\ty_s^i,\tV_s^i) ) ds\right)^2 +  2 \left( \int_0^t  \tfrac{1}{N}\sum_{i=1}^N  (\sigma^\alpha_y(y_s^i,V_s^i)  - \sigma^\alpha_y(\ty_s^i,\tV_s^i) ) dW_s^i \right)^2. 
\end{align*}
Notice that the processes $(\overline{Z}_t)$, defined  by 
$$ \overline{Z}_t := 2 \Big( \int_0^t  \tfrac{1}{N}\sum_{i=1}^N  (\sigma^\alpha_y(y_s^i,V_s^i)  - \sigma^\alpha_y(\ty_s^i,\tV_s^i) ) dW_s^i \Big)^2$$ 
is such that 
$\sup_{0\leq t\leq T}\EE (\overline{Z}_t)^2 \leq \tfrac{C}{N}. $
Since $b^\alpha_y$ and $\sigma^\alpha_y$ are Lipschitz on $[0,1]\times \er$, we get 
\begin{align*}
(\dyb_t)^2 \leq & \int_0^t  C ( \dyb_s)^2  ds + \int_0^t  C  \Big(\tfrac{1}{N}\sum_{i=1}^N  |V^i_s - \tV^i_s| \Big)^2 ds +\overline{Z}_t\\
\leq & \int_0^t  C  (\dyb_s)^2  ds + \int_0^t  \tfrac{C}{N}\sum_{i=1}^N  |R^{i,N}_s - \tR^i_s|^2 ds +\overline{Z}_t
\end{align*}
Combining again Gronwall's lemma  and integration by parts we obtain
\begin{align}\label{ineq:delta_y}
\dyb_t^2 \leq & \int_0^t  \tfrac{C}{N}\sum_{i=1}^N  |R^{i,N}_s - \tR^i_s|^2 ds  + \int_0^t C e^{C(t-s)} \Big(   \int_0^s   \tfrac{C}{N}\sum_{i=1}^N  |R^{i,N}_\theta - \tR^i_\theta|^2  d\theta\Big) ds +\overline{Z}_t  +  \int_0^t C e^{C(t-s)} \overline{Z}_s ds \nonumber \\
 \leq & C \int_0^t \tfrac{1}{N}\sum_{i=1}^N  |R^{i,N}_s - \tR^i_s|^2 ds  +\overline{Z}_t  +  \int_0^t C e^{C(t-s)} \overline{Z}_s ds. 
\end{align}

\paragraph{Step 3. The bound for $\EE\left[ \sup_{0\leq t\leq T} \sum_{\alpha \in \PP } 
|R^{i_\alpha,N}_t-{\tR}^{i_\alpha}_t|^2\right]$.} 

Combining the last inequality \eqref{ineq:delta_y} with \eqref{ineq:delta-R-recap}, it comes 
\begin{align*}
\dyb_t^2 & \leq 
C \int_0^t \int_0^s \Big(\tfrac{C}{N}\sum_{i=1}^N  \sum_{\gamma\in\PP } 
 (J^{i\gamma}_\theta)^2 
(\widetilde{V}^{i}_\theta  - \overline{V}^{\alpha\gamma})^2\Big)  (\dyb_\theta)^2 d\theta ds  +\int_0^t  \tfrac{C}{N}\sum_{i=1}^N  Z^{i}_s ds + \overline{Z}_t  +  \int_0^t C e^{C(t-s)} \overline{Z}_s ds\\
& \quad = C \int_0^t (t-s) \Big(\tfrac{C}{N}\sum_{i=1}^N  \sum_{\gamma\in\PP } 
 (J^{i\gamma}_s)^2 
(\widetilde{V}^{i}_s  - \overline{V}^{\alpha\gamma})^2\Big)  (\dyb_s)^2 ds  +\int_0^t  \tfrac{C}{N}\sum_{i=1}^N  Z^{i}_s ds + \overline{Z}_t  +  \int_0^t C e^{C(t-s)} \overline{Z}_s ds\\
& \quad = C \int_0^t (t-s)\EE\Big[ \sum_{\gamma\in\PP } 
 (J^{1\gamma}_s)^2 
(\widetilde{V}^{1}_s  - \overline{V}^{\alpha\gamma})^2\Big]  (\dyb_s)^2 ds  + \gamma_t +\int_0^t  \tfrac{C}{N}\sum_{i=1}^N  Z^{i}_s ds + \overline{Z}_t  +  \int_0^t C e^{C(t-s)} \overline{Z}_s ds \\
\end{align*}
where
$$\gamma_t  := C \int_0^t (t-s) \Big\{ (\tfrac{C}{N}\sum_{i=1}^N  \sum_{\gamma\in\PP } 
 (J^{i\gamma}_s)^2 
(\widetilde{V}^{i}_s  - \overline{V}^{\alpha\gamma})^2) - \EE[ \sum_{\gamma\in\PP } 
 (J^{1\gamma}_s)^2 
(\widetilde{V}^{1}_s  - \overline{V}^{\alpha\gamma})^2]\Big\}  (\dyb_s)^2 ds$$
is such that $\sup_{0\leq t\leq T}\EE (\gamma_t)^2 \leq \dfrac{C}{N}$, 
since $(\dyb_s)^2\leq 1$ a.s.. Taking the expectation of the last inequality, we get 
\begin{align*}
\EE [\dyb_t^2] & \leq  C \int_0^t (t-s)\EE\Big[ \sum_{\gamma\in\PP } 
 (J^{1\gamma}_s)^2 
(\widetilde{V}^{1}_s  - \overline{V}^{\alpha\gamma})^2\Big]  \EE[\dyb_s]^2 ds  + \dfrac{C}{\sqrt{N}} \\
&\leq  \dfrac{C}{\sqrt{N}} 
\end{align*}
by applying again Gronwall's lemma in the case of a non-decreasing remainder.  Coming back to \eqref{ineq:delta-R-recap},  we get 
\begin{align*}
\EE\Big[\tfrac{1}{N}\sum_{i=1}^N  |R^{i,N}_t - \tR^i_t|^2\Big]
 \leq & C \int_0^t  \EE[\sum_{\gamma\in\PP } 
 (J^{1\gamma}_s)^2 
(\widetilde{V}^{1}_s - \overline{V}^{\alpha\gamma})^2]  \EE[(\dyb_s)^2] ds  +   \tfrac{1}{N}\sum_{i=1}^N  \EE[Z^{i}_t] + \EE[\beta_t] 
\end{align*}
where again 
$$\beta_t  :=C \int_0^t \Big\{ (\tfrac{1}{N}\sum_{i=1}^N  \sum_{\gamma\in\PP } 
 (J^{i\gamma}_s)^2 
(\widetilde{V}^{i}_s  - \overline{V}^{\alpha\gamma})^2) - \EE[ \sum_{\gamma\in\PP } 
 (J^{1\gamma}_s)^2 
(\widetilde{V}^{1}_s  - \overline{V}^{\alpha\gamma})^2]\Big\}  (\dyb_s)^2 ds,$$
is such that $\sup_{0\leq t\leq T}\EE (\beta_t)^2 \leq \dfrac{C}{N}$.
Using \eqref{hypo:proportion}, this ends the proof of the proposition. 
\end{proof}

\section{Conclusion}
In this note we have set the work published in \cite{baladron-al} on a totally rigorous footing. In doing so we also have shed some new light on the way to incorporate noise in the ion channels equations for the Hodgkin-Huxley model and in the amount of neurotransmitters at the synapses in both the Hodgkin-Huxley and Fitzhugh-Nagumo models. 

The techniques in this paper could be extended to a more generic form of interaction kernel $k[r;\mu]$ in \eqref{eq:kernel-k}. 
Notice also that the hypothesis \ref{hypo:coeff}-(iii) should allow to prove the convergence in time to equilibrium of the mean field limits.  

\subsection*{Acknowledgements}
This work was partially supported by the European Union Seventh Framework Programme (FP7/2007-2013) under grant agreement no. 269921 (BrainScaleS), no. 318723 (Mathemacs), and by the ERC advanced grant NerVi no. 227747.

\bibliographystyle{plain}

\begin{thebibliography}{10}
\bibitem{alfonsi}
A. Alfonsi. On the discretization schemes for the CIR (and Bessel squared) processes.
\emph{ Monte Carlo Methods and Appl.}, 11(4): 355-384, 2005.

\bibitem{baladron-al}
J. Baladron, D. Fasoli, O. Faugeras, and J. Touboul.
Mean Field description and propagation of chaos in networks 
of Hodgkin-Huxley and Fitzhugh-Nagumo neurons.
\emph{Journal Mathematical Neuroscience}, 2(10), 2012.

\bibitem{K-S}
I. Karatzas and S.E. Shreve.
\emph{Brownian Motion and Stochastic Calculus.} 
Graduated Texts in Mathematics 113.  Springer-Verlag, 1988.

\bibitem{L-S}
E. Lu{\c{c}}on and W. Stannat.
{Mean field limit for disordered diffusions with singular interactions.}
\emph{Ann. Appl. Probab.} 24(5), 1946–1993, 2014.


\bibitem{McKean}
H.P. McKean.
Propagation of chaos for a class of non-linear parabolic equations. 
in \emph{1967 Stochastic Differential Equations}.
Lecture Series in Differential Equations, Session 7, Catholic Univ., 
\emph{Air Force Office Sci. Res., Arlington, Va.}, 41-57, 1967.

\bibitem{meleard}
S. M{\'e}l{\'e}ard. 
{Asymptotic behaviour of some interacting particle systems; McKean-Vlasov and Boltzmann models} in \emph{Probabilistic Models for Nonlinear Partial Differential Equations}, {Lecture Notes in Mathematics},{1627}, {Springer Berlin Heidelberg}, 1996.

\bibitem{protter}
P.E. Protter.
\emph{Stochastic Integration and Differential Equations}.
Stochastic Modelling and Applied Probability 21,
Springer-Verlag, 2005.

\bibitem{R-Y}
D. Revuz and M. Yor.
\emph{Continuous Martingales and Brownian Motion}. Springer,
1999.

\bibitem{scheutzow}
M. Scheutzow.
Uniqueness and nonuniqueness of solutions of Vlasov-McKean equations. 
\emph{J. Austral. Math. Soc. Ser. A }
43(2), 246-256, 1987.

\bibitem{sznitman}
A.S. Sznitman.
Topics in propagation of chaos.
\emph{Ecole d'Et\'e de Probabilit\'es de
Saint Flour XIX},  P.L. Hennequin (Ed.), Lecture Notes Math, 1464, 
Springer, 1989.

\bibitem{touboul:14}
J.~Touboul.
\newblock The propagation of chaos in neural fields.
\newblock {\em Ann. Appl. Probab.} 24(3), 1298–1328, 2014.

\bibitem{wainrib:06}
G. Wainrib. 
\emph{Randomness in Neurons: A Multiscale Probabilistic Analysis.}
Ph.D. thesis, \'Ecole Polytechnique (Paris, France), 2010.
\end{thebibliography}

\end{document}